\documentclass{mfat}
\pagespan{48}{61}

\newtheorem{thm}{Theorem}[section]

\newtheorem{lem}[thm]{Lemma}
\newtheorem{rem}[thm]{Remark}

\newtheorem{assumption}[thm]{Assumption}

\numberwithin{equation}{section}

\begin{document}

\title[On the finiteness of the discrete spectrum of a $3 \times 3$ operator matrix]
      {On the finiteness of the discrete spectrum of a $3 \times 3$ operator matrix}

\author[Tulkin H. Rasulov]{Tulkin H. Rasulov}
\address{Faculty of Physics and Mathematics, Bukhara State University, 11 M. Ikbol
str., Bukhara, 200100, Uzbekistan}
\email{rth@mail.ru}

\date{01/02/2014}

\begin{abstract}
  An operator matrix $H$ associated with a lattice system describing
  three particles in interactions, without conservation of the number
  of particles, is considered.  The structure of the essential spectrum
  of $H$ is described by the spectra of two families of the
  generalized Friedrichs models.  A symmetric version of the Weinberg
  equation for eigenvectors of $H$ is obtained. The conditions which
  guarantee the finiteness of the number of discrete eigenvalues
  located  below the bottom of the three-particle branch of
  the essential spectrum of $H$ is found.
\end{abstract}

\maketitle

\section{Introduction}

One of important problems in the spectral theory of Schr\"{o}dinger
operators and Hamiltonians (operator matrices) in a Fock space is to
study the number of eigenvalues (bound states) located outside the
essential spectrum. The first mathematical result on the finiteness of
the discrete spectrum of Schr\"{o}dinger operators for general
interactions was obtained by Uchiyama in \cite{Uch69}. Under natural
assumptions on the potential, the essential spectrum of the continuous
Schr\"{o}dinger operator $H_{\rm c}$ of a system of three pair-wise
interacting particles coincides with the half-axis $[\kappa; \infty),$
$\kappa \leq 0.$ In independent investigations of Yafaev \cite{Yaf72}
and Zhislin \cite{Zhis72}, it was shown that for $\kappa<0$ and a
sufficiently rapid decrease of the interactions in the coordinate
space representation the discrete spectrum of $H_{\rm c}$ is actually
finite. In the case $\kappa=0$ the finiteness of the discrete spectrum
of $H_{\rm c}$ with certain decreasing interactions was established by
Yafaev \cite{Yaf75}. Yafaev's results are based on the investigation
of the Faddeev and Weinberg type system of integral equations for the
resolvent.

The problem of finiteness of the number of eigenvalues of the
three-particle discrete Schr\"{o}dinger operators $H_{\rm d}$ was
studied by many authors, see for example, \cite{AL, LM, M08}.  The
authors of \cite{AL} used the Faddeev and Weinberg type equations and
an expansion of the Fredholm determinant to prove finiteness of the
discrete spectrum of $H_{\rm d}$ with pair contact interactions when
the corresponding two-particle discrete Schr\"{o}dinger operators have
no virtual levels. The Birman-Schwinger principle was used in
\cite{LM} to prove that the discrete spectrum of the operator $H_{\rm
  d}$ describing systems of three particles (two bosons and a third
particle of a different nature) is finite.  In \cite{LS}, applying the
methods developed in \cite{Yaf75} to the Hamiltonian $H_{\rm d}$ of a
system of three arbitrary particles on a lattice, finiteness of the
discrete spectrum of $H_{\rm d}$ is proved if either only one or none
of the two-particle subsystems has a virtual level.  In \cite{M08},
the finiteness of the number of eigenvalues of $H_{\rm d}$ with a
specific class of potentials is proved where one of the particles has
an infinite mass.

In all of the above mentioned papers devoted to the finiteness of the
discrete spectrum, it was considered systems with a fixed number of
quasi-particles.  It is worth to mention that there are important
problems in the theory of solid-state physics \cite{Mog}, quantum
field theory \cite{Frid}, statistical physics \cite{Mal-Min, MS},
fluid mechanics \cite{Cha61}, magnetohydrodynamics \cite{Lif89} and
quantum mechanics \cite{Tha92} where the number of quasi-particles is
finite but not fixed. Recall that the study of systems describing $n$
particles in interaction without conservation of the number of
particles can be reduced to the investigation of the spectral
properties of self-adjoint operators acting in the $n$-particle cut
subspace of the Fock space \cite{Frid, MS, Mog, SSZ}. In \cite{SSZ},
geometric and commutator techniques were developed in order to find
the location of the spectrum and to prove absence of singular
continuous spectrum for Hamiltonians without the conservation of
particle number.

In the present paper we consider an operator matrix $H$ associated
with the lattice system describing three particles in interactions
without conservation of the number of particles. This operator acts in
a three-particle subspace ${\mathcal H}$ of the bosonic Fock space and
it is a lattice analogue of the spin-boson Hamiltonian \cite{MS}.  We
find sufficient conditions for the finiteness of the discrete spectrum
of $H.$ Note that the operator matrix $H$ has been considered before
in \cite{MR, Ras08, Ras10-1, YoM} where only its essential spectrum
was investigated.

The organization of the present paper is as follows. Section 1 is
an introduction to the whole work. In Section 2, the operator matrix
$H$ is described as a bounded self-adjoint operator in ${\mathcal H}$
and the main results are formulated. In Section 3, we prove some auxiliary lemmas.
In Section 4, we obtain a symmetric version of the Weinberg equation
for eigenvectors of $H.$ Section 5 is devoted to the proof of the
main results.

\section{The operator matrix and main results}

\subsection{The operator matrix}
Let ${\Bbb C},$ ${\Bbb R}$ and ${\Bbb Z}$ be the set of all complex,
real and integer numbers, respectively. We denote by ${\Bbb T}^3$ the
three-dimensional torus (the first Brillouin zone, i.e., the dual
group of ${\Bbb Z}^3$), the cube $(-\pi,\pi]^3$ with appropriately
identified sides is equipped with its Haar measure. The torus ${\Bbb
  T}^3$ will always be considered as an Abelian group with respect to
the addition and multiplication by real numbers regarded as operations
on the three-dimensional space ${\Bbb R}^3$ modulo $(2 \pi {\Bbb
  Z})^3.$

Let $ L_2({\Bbb T}^3)$
be the Hilbert space of square integrable (complex) functions
defined on ${\Bbb T}^3$ and $ L_2^{\rm s} (({\Bbb T}^3)^2)$ be the
Hilbert space of square integrable (complex) symmetric functions
defined on $({\Bbb T}^3)^2.$ Denote by ${\mathcal H}$ the direct
sum of spaces ${\mathcal H}_1={\Bbb C},$ ${\mathcal H}_1=L_2({\Bbb
  T}^3)$ and ${\mathcal H}_2=L_2^{\rm s}(({\Bbb T}^3)^2),$ that is,
${\mathcal H}={\mathcal H}_0 \oplus {\mathcal H}_1 \oplus
{\mathcal H}_2.$

Let us consider the operator matrix (Hamiltonian) $H$ acting in
the Hilbert space ${\mathcal H}$ as
$$
H=\left( \begin{array}{ccc}
    H_{00} & H_{01} & 0\\
    H_{01}^* & H_{11} & H_{12}\\
    0 & H_{12}^* & H_{22}\\
  \end{array}
\right),
$$
where the entries $H_{ij}: {\mathcal H}_j \to {\mathcal H}_i,$ $i \leq j,$
$i,j=0,1,2$ are defined by
$$
H_{00}f_0=w_0 f_0,\quad H_{01}f_1=\int_{{\Bbb T}^3}
v_0(s)f_1(s)\,ds, \quad (H_{11}f_1)(p)=w_1(p)f_1(p),
$$
$$
(H_{12}f_2)(p)= \int_{{\Bbb T}^3} v_1(s) f_2(p,s)\,ds, \quad
H_{22}=H_{22}^0-V, \quad (H_{22}^0f_2)(p,q)=w_2(p,q)f_2(p,q),
$$
$$
(Vf_2)(p,q)=v_2(q) \int_{{\Bbb T}^3} v_2(s)
f_2(p,s)\,ds+v_2(p) \int_{{\Bbb T}^3} v_2(s) f_2(s,q)\,ds.
$$

Here $f_i \in {\mathcal H}_i,$ $i=0,1,2;$ $w_0$ is a fixed real
number, $w_1(\cdot)$ and $v_i(\cdot),$ $i=0,1,2$ are real-valued
continuous functions on ${\Bbb T}^3,$ the function
$w_2(\cdot,\cdot)$ is a real-valued continuous symmetric function
on $({\Bbb T}^2)^2.$
The operator $H_{ij}^*$ $(i<j)$ denotes the adjoint to $H_{ij}$
and
$$
(H_{01}^* f_0)(p)=v_0(p)f_0, \quad (H_{12}^*
f_1)(p,q)=\frac{v_1(p)f_1(q)+v_1(q)f_1(p)}{2},\quad f_i \in
{\mathcal H}_i,\quad i=0,1.
$$

It follows that under these assumptions $H$ is bounded and
self-adjoint.

We recall that the operators $H_{01}$ and $H_{12}$ (resp.
$H_{01}^*$ and $H_{12}^*$) are called annihilation (resp. creation)
operators, respectively. In the present paper we consider the case
where the number of annihilations and creations of the particles
of the system is equal to 1, that is,
$H_{ij}\equiv 0$ for all $|i-j|>1.$

It is known that the three-particle discrete Schr\"{o}dinger
operator $\widehat{H}$ in the momentum representation acts on the
Hilbert space $L_2(({\Bbb T}^3)^3).$ Introducing the total
quasi-momentum $K \in {\Bbb T}^3$ and choosing relative coordinate system,
we decompose $\widehat{H}$ into the von Neumann direct integral (see for example
\cite{AL, LM, LS, M08})
$$
\widehat{H}=\int_{{\Bbb T}^3} \widehat{H}(K)\,dK,
$$
where the bounded self-adjoint operator $\widehat{H}(K),$ $K \in
{\Bbb T}^3,$ acts on the Hilbert space $L_2(\Gamma_K).$ Here
$\Gamma_K \subset ({\Bbb T}^3)^2$ being some manifold.

Notice that the operator matrix $H$ satisfies the main spectral
properties of the three-particle discrete Schr\"{o}dinger operator
$\widehat{H}(0),$ where the role of two-particle discrete
Schr\"{o}\-dinger operators is played by the family of the generalized
Friedrichs models \cite{ALR, ALR1}. For this reason the Hilbert space
${\mathcal H}$ is called the {\it three-particle cut subspace} of
the bosonic Fock space ${\mathcal F}_{\rm s}(L_2({\Bbb T}^3))$
over $L_2({\Bbb T}^3)$ and the operator matrix $H$ is associated
to a system describing three particles in interaction without
conservation of the number of particles. The operator $H_{22}$ is
associated to a system of three quantum particles on a lattice.

To formulate the main results of the paper we introduce the
operators $H_1$ and $H_2$ acting in the Hilbert spaces ${\mathcal
  H}$ and ${\mathcal H}_2,$ respectively, as
$$
H_1:=\left( \begin{array}{ccc}
    H_{00} & H_{01} & 0\\
    H_{01}^* & H_{11} & H_{12}\\
    0 & H_{12}^* & H_{22}^0\\
  \end{array}
\right), \quad H_2:=H_{22},
$$
and the family of bounded self-adjoint operators (generalized
Friedrichs models) $h(p),$ $p\in {\Bbb T}^3,$ acting in
${\mathcal H}_0 \oplus {\mathcal H}_1$ as
$$
h(p)=\left( \begin{array}{cc}
    h_{00}(p) & h_{01}\\
    h_{01}^* & h_{11}(p)\\
  \end{array}
\right),
$$
where
$$
h_{00}(p)f_0=w_1(p)f_0,\quad h_{01}f_1=\frac{1}{\sqrt{2}}
\int_{{\Bbb T}^3} v_1(s)f_1(s)\,ds,
$$
$$
h_{11}(p)=h_{11}^0(p)-v, \quad (h_{11}^0(p)f_1)(q)=w_2(p,q)f_1(q),
\quad (vf_1)(q)= v_2(q) \int_{{\Bbb T}^3} v_2(s) f_1(s)\,ds.
$$

We recall that the operator $h(p)$ is also called molecular-resonance model and it is associated
with the Hamiltonian of the system  consisting of at most two particles on the three-dimensional
lattice, interacting via both a nonlocal potential and creation and annihilation operators.

In \cite{Ras10-1} it was shown that for any $p \in {\Bbb T}^3$ the
operator $h(p)$ has at most three eigenvalues.

The spectrum, the essential spectrum, the
discrete and point spectrum of a bounded self-adjoint operator will be
denoted by $\sigma(\cdot),$ $\sigma_{\rm ess}(\cdot),$
$\sigma_{\rm disc}(\cdot)$ and $\sigma_{\rm p}(\cdot)$ respectively.

Set
$$
m:=\min\limits_{p,q \in {\Bbb T}^3} w_2(p,q), \quad
M:=\max\limits_{p,q \in {\Bbb T}^3} w_2(p,q).
$$

The following theorem describes the location of the
essential spectrum of the operator $H$ by the spectrum of the
family $h(p)$ of the generalized Friedrichs models \cite{Ras10-1}.

\begin{thm}\label{THM1} For the essential spectrum of $H$
  the following equality holds:
  \begin{equation}\label{Eq1}
    \sigma_{\rm ess}(H)=\sigma \cup [m; M], \quad \sigma:=
    \bigcup\limits_{p\in {{\Bbb T}^3}}\sigma_{\rm disc}(h(p)).
  \end{equation}
  Moreover, the set $\sigma_{\rm ess}(H)$ is a union of at
  most four bounded closed intervals.
\end{thm}

The subsets $\sigma$ and $[m; M]$ are called two-particle and
three-particle branches of the essential spectrum of $H,$
respectively.

\subsection{Main assumptions}
From now on we always assume that $\{\alpha, \beta\}=\{1,2\}$ and
$\alpha \neq \beta.$ Denote $\bar{\pi}:=(\pi,\pi,\pi).$

\begin{assumption}\label{A1}
  The function $v_\alpha(\cdot)$ is $2\bar{\pi}$ periodic and
  $v_\beta(\cdot)$ satisfies the
  condition
  \begin{equation}\label{Eq2}
    \int_{{\Bbb T}^3} v_\beta(s)g(s)\,ds=0
  \end{equation}
  for any $2\bar{\pi}$ periodic function $g\in L_2({\Bbb T}^3).$
\end{assumption}

\begin{assumption}\label{A2}
  {\rm (i)} The function $w_2(\cdot,\cdot)$ is $2\bar{\pi}$ periodic on each
  variable $p$ and $q,$ that is,
  $w_2(p+2\bar{\pi},q)=w_2(p,q+2\bar{\pi})=w_2(p,q)$
  for all $p,q \in {\Bbb T}^3;$\\
  {\rm (ii)} The function $w_2(\cdot,\cdot)$ has a unique
  non-degenerate minimum at the point $(p_0,p_0) \in ({\Bbb T}^3)^2.$
  All third order partial derivatives of the functions $w_1(\cdot)$ and $w_2(\cdot,\cdot)$ are
  continuous on ${\Bbb T}^3$ and $({\Bbb T}^3)^2,$ respectively.
\end{assumption}

Under the Assumption \ref{A1} and the part (i) of Assumption
\ref{A2} the discrete spectrum of $h(p)$ coincides (see Lemma
\ref{LEM1} below) with the union of discrete spectra of the
operators
$$
h_1(p):=\left(
  \begin{array}{cc}
    h_{00}(p) & h_{01}\\
    h_{01}^* & h_{11}^0(p)\\
  \end{array}
\right) \quad \mbox{and} \quad h_{2}(p): = h_{11}(p).
$$

It follows from the definition of the operator $h_\alpha(p)$ that
its structure is simpler than that of $h(p).$ Using the Weyl
theorem one can easily show that
$$
\sigma_{\rm ess}(h(p))=\sigma_{\rm ess}(h_1(p))=\sigma_{\rm
  ess}(h_2(p))=[m(p); M(p)],
$$
where the numbers $m(p)$ and $M(p)$ are defined by
$$
m(p):= \min_{q\in {\Bbb T}^3} w_2(p,q),\quad M(p):= \max_{q\in
  {\Bbb T}^3} w_2(p,q).
$$

For any fixed $p\in {\Bbb T}^3,$ we define the analytic functions
in ${\Bbb C}\setminus [m(p); M(p)]$ by
$$
\Delta_1(p\,; z):=w_1(p)-z-\frac{1}{2} \int_{{\Bbb T}^3}
\frac{v_1^2(s)\,ds}{w_2(p,s)-z}, \quad \Delta_2(p\,; z):=
1-\int_{{\Bbb T}^3} \frac{v_2^2(s)\,ds}{w_2(p,s)-z},
$$
which are Fredholm determinants associated with the
operators $h_1(p)$ and $h_2(p),$ respectively.

Since the function $w_2(\cdot,\cdot)$ has a unique non-degenerate
minimum at $(p_0, p_0) \in ({\Bbb T}^3)^2$ and the
function $v_\alpha(\cdot)$ is a continuous on ${\Bbb T}^3,$ for
any $p \in {\Bbb T}^3$ the integral
$$
\int_{{\Bbb T}^3} \frac{v_\alpha^2(s)\,ds}{w_2(p, s)-m}
$$
is positive and finite. Then the Lebesgue dominated convergence theorem
yields $\Delta_\alpha(p_0\,; m)=\lim\limits_{p\to p_0}
\Delta_\alpha(p\,; m),$ and hence the function
$\Delta_\alpha(\cdot\,; m)$ is a continuous on ${\Bbb T}^3.$

Note that using the fact \cite{ALDj, ALR}
$$
\sigma_{\rm ess}(H_\alpha)=\sigma_\alpha \cup [m; M], \quad
\sigma_\alpha: = \bigcup\limits_{p\in {{\Bbb T}^3}}\sigma_{\rm
  disc}(h_\alpha(p))
$$
together with Assumption \ref{A1} and part (i) of Assumption \ref{A2} the equality
\eqref{Eq1} can be written as
\begin{equation}\label{Eq12}
  \sigma_{\rm ess}(H)=\sigma_{\rm
    ess}(H_1) \cup \sigma_{\rm ess}(H_2).
\end{equation}

It was shown in \cite{ALDj, ALR} that if $\min\limits_{p \in {\Bbb T}^3}\Delta_\alpha(p\,;
m)<0,$ then $ \sigma_\alpha \cap
(-\infty; m] \neq \emptyset.$ Assuming $\min\limits_{p \in {\Bbb
    T}^3}\Delta_\alpha(p\,; m)<0,$ we introduce the following numbers:

$$
E_{\rm min}^{(\alpha)}:=\min \left\{ \sigma_\alpha \cap (-\infty;
  m] \right\},\quad E_{\rm max}^{(\alpha)}:=\max \left\{
  \sigma_\alpha \cap (-\infty; m] \right\}.
$$

The following theorem \cite{ALDj, Ras10-1} describes the structure of
the part of the essential spectrum of $H_\alpha$ located in
$(-\infty; M].$

\begin{thm}\label{THM2} Let part {\rm (ii)} of Assumption $\ref{A2}$ be fulfilled.
  Then the following assertions hold.\\
  {\rm (i)} If $\min\limits_{p \in {\Bbb T}^3}\Delta_\alpha(p\,;
  m)\geq 0,$ then
  $$
  (-\infty; M] \cap \sigma_{\rm ess}(H_\alpha)=[m; M].
  $$
  {\rm (ii)} If $\min\limits_{p \in {\Bbb T}^3}\Delta_\alpha(p\,;
  m)<0$ and $\max\limits_{p \in {\Bbb T}^3}\Delta_\alpha(p\,; m)\geq
  0,$ then
  $$
  (-\infty; M] \cap \sigma_{\rm ess}(H_\alpha)=[E_{\rm
    min}^{(\alpha)}; M], \quad E_{\rm min}^{(\alpha)}<m.
  $$
  {\rm (iii)} If $\max\limits_{p \in {\Bbb T}^3}\Delta_\alpha(p\,;
  m)<0,$ then
  $$
  (-\infty; M] \cap \sigma_{\rm ess}(H_\alpha)=[E_{\rm
    min}^{(\alpha)}; E_{\rm max}^{(\alpha)}] \cup [m; M], \quad E_{\rm
    max}^{(\alpha)}<m.
  $$
\end{thm}

We notice that if Assumption \ref{A1} and part (i) of Assumption
\ref{A2} hold, then Theorem \ref{THM2} together with the equality
\eqref{Eq12} describes the structure of the part of the essential
spectrum of $H$ located in $(-\infty; M].$

If $\min\limits_{p \in {\Bbb T}^3}\Delta_\alpha(p\,; m)<0,$ then from
$E_{\rm min}^{(\alpha)}, E_{\rm max}^{(\alpha)} \in
\sigma_\alpha$ it follows that there exist positive integers $n_\alpha,$ $k_\alpha$
and points $\{p_{\alpha i}\}_{i=1}^{n_\alpha},
\{q_{\alpha j}\}_{j=1}^{k_\alpha} \subset {\Bbb T}^3$ such that
$$
\{p \in {\Bbb T}^3: \Delta_\alpha(p\,; E_{\rm min}^{(\alpha)})=0
\}=\{p_{\alpha 1}, \ldots, p_{\alpha n_\alpha}\},
$$
$$
\{p \in {\Bbb T}^3: \Delta_\alpha(p\,; E_{\rm max}^{(\alpha)})=0
\}=\{q_{\alpha 1}, \ldots, q_{\alpha k_\alpha}\}.
$$

\begin{assumption}\label{A3}
  There exist positive numbers $C,$ $\delta$ and $\beta_{\alpha i} \in (0;
  2],$ $i=1, \ldots, n_\alpha$ such that
  $$
  |\Delta_\alpha(p\,; E_{\rm min}^{(\alpha)})|\geq C |p-p_{\alpha
    i}|^{\beta_{\alpha i}}, \quad p \in U_\delta(p_{\alpha i}), \quad i=1,
  \ldots, n_\alpha,
  $$
  and the inequality $\Delta_\alpha(p\,; E_{\rm min}^{(\alpha)})>0$
  holds for all $p \in {\Bbb T}^3 \setminus \{p_{\alpha 1}, \ldots,
  p_{\alpha n_\alpha}\}.$
\end{assumption}

\begin{assumption}\label{A4}
  There exist positive numbers $K,$ $\rho$ and $\gamma_{\alpha j} \in (0;
  2],$ $j=1, \ldots, k_\alpha$ such that
  $$
  |\Delta_\alpha(p\,; E_{\rm max}^{(\alpha)})|\geq K |p-q_{\alpha
    j}|^{\gamma_{\alpha j}}, \quad p \in U_\rho(q_{\alpha j}), \quad j=1,
  \ldots, k_\alpha,
  $$
  and the inequality $\Delta_\alpha(p\,; E_{\rm max}^{(\alpha)})<0$
  holds for all $p \in {\Bbb T}^3 \setminus \{q_{\alpha 1}, \ldots,
  q_{\alpha k_\alpha}\}.$
\end{assumption}

\subsection{Statement of the main results}
Here we formulate main results of the paper.

\begin{thm}\label{THM3} Let part {\rm
    (i)} of Assumption $\ref{A2}$ be fulfilled.\\
  {\rm (i)} If Assumption $\ref{A1}$ holds with $\alpha=1$ and
  in addition, the functions $v_0(\cdot),$ $w_1(\cdot)$
  are $2\bar{\pi}$ periodic, then $\sigma_{\rm
    disc}(H_1) \subset \sigma_{\rm p}(H).$\\
  {\rm (i)} If Assumption $\ref{A1}$ holds with $\alpha=2,$ then
  $\sigma_{\rm disc}(H_2) \subset \sigma_{\rm p}(H).$
\end{thm}

\begin{thm}\label{THM4} Let Assumptions $\ref{A1}$ and
  $\ref{A2}$
  be fulfilled. Assume\\
  {\rm ($\alpha$.1)} $\min\limits_{p \in {\Bbb T}^3} \Delta_\alpha(p\,; m)>0;$\\
  {\rm ($\alpha$.2)} $\min\limits_{p \in {\Bbb T}^3}
  \Delta_\alpha(p\,; m)<0,$ $\max\limits_{p \in {\Bbb
      T}^3}\Delta_\alpha(p\,; m)\geq 0$ and Assumption $\ref{A3}$ holds;\\
  {\rm ($\alpha$.3)} $\max\limits_{p \in {\Bbb T}^3}
  \Delta_\alpha(p\,; m)<0$ and Assumptions $\ref{A3}, \ref{A4}$
  hold.

  If for some $i,j \in \{1,2,3\}$ the conditions $(1.i)$ and $(2.j)$
  hold, then the operator matrix $H$ has a finite number of discrete eigenvalues
  lying on the left of $m.$
\end{thm}

\begin{rem}
  The class of functions $w_1(\cdot),$ $v_i(\cdot),$ $i=1,2$ and
  $w_2(\cdot,\cdot)$ satisfying the conditions in Theorem
  $\ref{THM4}$ is nonempty $($see Lemma $\ref{Example}).$
\end{rem}

\begin{rem}
  Note that comparing Theorems $\ref{THM3}$ and $\ref{THM4}$ we have that if the condition
  $(\alpha.j)$ in Theorem $\ref{THM4}$ holds for some $j \in \{1,2,3\},$
  then the operator $H_\alpha$ has a finite number of discrete eigenvalues lying on the left of $m.$
  If $\min\limits_{p \in {\Bbb T}^3} \Delta_\alpha(p\,; m)=\Delta_\alpha(p_0\,; m)=0$
  and $v_\alpha(p_0) \neq 0,$ then $\min \sigma_{\rm ess}(H_\alpha)=m$ and
  it was shown in \cite{ALR} for $\alpha=1$ and in \cite{ALDj} for $\alpha=2$ that the operator
  $H_\alpha$ has infinitely many eigenvalues lying on the left of $m.$ Hence, in this case by
  Theorem $\ref{THM3}$ the operator $H$ also has infinitely many eigenvalues lying on the left of $m.$
\end{rem}

\section{Some auxiliary statements}

The following lemma describes the relation between the eigenvalues
of the operators $h(p)$ and $h_\alpha(p).$

\begin{lem}\label{LEM1} Let Assumption $\ref{A1}$ and part {\rm
    (i)} of Assumption $\ref{A2}$ be fulfilled. For any fixed $p \in
  {\Bbb T}^3$ the number $z(p) \in {\Bbb C} \setminus [m(p); M(p)]$
  is an eigenvalue for $h(p)$ if and only if $z(p)$ is an eigenvalue
  for at least one of the operators $h_1(p)$ and $h_2(p).$
\end{lem}

\begin{proof} Let $p \in {\Bbb T}^3$ be fixed.
  Suppose $(f_0,f_1) \in {\mathcal H}_0 \oplus {\mathcal H}_1$ is an
  eigenvector of the operator $h(p)$ associated with the eigenvalue
  $z(p) \in {\Bbb C} \setminus [m(p); M(p)]$. Then $f_0$ and $f_1$
  satisfy the following system of equations:
  \begin{equation}\label{Eq3}
    \begin{gathered}
      (w_1(p)-z(p))f_0+ \frac{1}{\sqrt{2}} \int_{{\Bbb T}^3}
      v_1(s)f_1(s)\,ds=0,
      $$
      \\
      \frac{1}{\sqrt{2}} v_1(q)f_0+(w_2(p,q)-z(p))f_1(q)-v_2(q)
      \int_{{\Bbb T}^3} v_2(s) f_1(s)\,ds=0.
    \end{gathered}
  \end{equation}

  Since for any $q\in {\Bbb T}^3$ the relation $w_2(p,q)-z(p) \neq 0$
  holds, from the second equation in the system \eqref{Eq3} for
  $f_1$ we have
  \begin{equation}\label{Eq4}
    f_1(q)=\frac{C_{f_1} v_2(q)}{w_2(p,q)-z(p)}- \frac{1}{\sqrt{2}}
    \frac{v_1(q) f_0}{w_2(p,q)-z(p)},
  \end{equation}
  where
  \begin{equation}\label{Eq5}
    C_{f_1}=\int_{{\Bbb T}^3} v_2(s) f_1(s)\,ds.
  \end{equation}

  Substituting the expression \eqref{Eq4} for $f_1$ into the first
  equation of the system \eqref{Eq3} and the equality \eqref{Eq5},
  we conclude that the system of equations \eqref{Eq3} has a
  nontrivial solution if and only if the system of equations
  \begin{align*}
    & \Delta_1(p\,; z(p))f_0+\frac{1}{\sqrt{2}} \int_{{\Bbb T}^3}
    \frac{v_1(s) v_2(s)\,ds}{w_2(p,s)-z(p)} C_{f_1} =0,\\
    & \frac{1}{\sqrt{2}} \int_{{\Bbb T}^3} \frac{v_1(s)
      v_2(s)\,ds}{w_2(p,s)-z(p)} f_0+\Delta_2(p\,; z(p)) C_{f_1}=0
  \end{align*}
  has a nontrivial solution $(f_0,C_{f_1}) \in {\Bbb C}^2,$ i.e. if the condition
  $$
  \Delta_1(p\,; z(p)) \Delta_2(p\,; z(p))-\frac{1}{2}
  \Bigg(\,\int_{{\Bbb T}^3} \frac{v_1(s) v_2(s)\,ds}{w_2(p,s)-z(p)}\,\Bigg)^2=0
  $$
  is satisfied.

  By part (i) of Assumption \ref{A2} for any fixed $p\in {\Bbb T}^3$
  the function $(w_2(p,\cdot)-z(p))^{-1} \in L_2({\Bbb T}^3)$ is $2 \bar{\pi}$
  periodic. Applying Assumption \ref{A1} we obtain
  $$
  \int_{{\Bbb T}^3} \frac{v_1(s) v_2(s)\,ds}{w_2(p,s)-z(p)}=0.
  $$

  If we set $v_2(q) \equiv 0$ in the operator $h(p),$ then $h(p)=h_1(p);$
  in this case the number $z(p) \in {\Bbb C} \setminus [m(p); M(p)]$ is an eigenvalue of
  $h_1(p)$ if and only if $\Delta_1(p\,; z(p))=0.$ Similarly one can show that the
  number $z(p) \in {\Bbb C} \setminus [m(p); M(p)]$ is an eigenvalue of
  $h_2(p)$ if and only if $\Delta_2(p\,; z(p))=0.$
  The lemma is proved.
\end{proof}

\begin{lem}\label{LEM2}
  Let $\min\limits_{p \in {\Bbb T}^3} \Delta_\alpha(p\,; m)>0.$ Then
  there exists a positive number $C_1$ such that the inequality
  $\Delta_\alpha(p\,; z) \geq C_1$ holds for all $p \in {\Bbb T}^3$
  and $z \leq m.$
\end{lem}

\begin{proof}
  Since for any $p \in {\Bbb T}^3$ the function $\Delta_\alpha(p\,;
  \cdot)$ is monotonically decreasing in $(-\infty; m],$ we have
  $$
  \Delta_\alpha(p\,; z) \geq \Delta_\alpha(p\,; m) \geq
  \min\limits_{p \in {\Bbb T}^3} \Delta_\alpha(p\,; m)>0
  $$
  for all $p \in {\Bbb T}^3$ and $z \leq m.$ Now setting
  $C_1:=\min\limits_{p \in {\Bbb T}^3} \Delta_\alpha(p\,; m)$ we
  complete the proof of lemma.
\end{proof}

For some $\delta>0$ we set
$$
U_\delta(p_0):=\{p\in {\Bbb T}^3: |p-p_0|<\delta\}.
$$

\begin{lem}\label{LEM6}
  If Assumption $\ref{A3}$ resp. $\ref{A4}$ holds, then for any
  $\delta>0$ there exist the positive numbers $C_1(\delta)$ and
  $C_2(\delta)$ such that\\
  {\rm (i)} $\Delta_\alpha(p\,; E_{\rm min}^{(\alpha)}) \geq
  C_1(\delta)$ for any $p \in {\Bbb T}^3 \setminus
  \bigcup\limits_{i=1}^{n_\alpha} U_\delta(p_{\alpha i});$\\
  resp.\\
  {\rm (ii)} $|\Delta_\alpha(p\,; E_{\rm max}^{(\alpha)})| \geq
  C_2(\delta)$ for any $p \in {\Bbb T}^3 \setminus
  \bigcup\limits_{j=1}^{k_\alpha} U_\delta(q_{\alpha j}).$
\end{lem}

\begin{proof}
  Let Assumption \ref{A3} be fulfilled. Then the inequality
  $\Delta_\alpha(p\,; E_{\rm min}^{(\alpha)})>0$ holds for any
  ${\Bbb T}^3 \setminus \{p_{\alpha 1}, \ldots, p_{\alpha
    n_\alpha}\}.$ Since for any $\delta>0$ the set ${\Bbb T}^3
  \setminus \bigcup\limits_{i=1}^{n_\alpha} U_\delta(p_{\alpha i})$
  is compact and $\Delta_\alpha(\cdot\,; E_{\rm min}^{(\alpha)})$ is
  the positive continuous function on this set, there exists the
  number $C_1(\delta)>0$ such that the assertion (i) of lemma holds.
  Proof of assertion (ii) is similar.
\end{proof}

\begin{lem}\label{LEM3}
  Let part {\rm (ii)} of Assumption $\ref{A2}$ be fulfilled. Then
  there exist positive numbers $C_1, C_2, C_3$ and $\delta$ such
  that the following inequalities hold:\\
  {\rm (i)} $C_1(|p-p_0|^2+|q-p_0|^2) \leq w_2(p,q)-m \leq
  C_2(|p-p_0|^2+|q-p_0|^2),$ $p,q \in U_\delta(p_0);$\\
  {\rm (ii)} $w_2(p,q)-m \geq C_3,$ $(p,q) \not\in U_\delta(p_0)
  \times U_\delta(p_0).$
\end{lem}

\begin{proof}
  By part {\rm (ii)} of Assumption $\ref{A2}$ the all third order
  partial derivatives of $w_2(\cdot,\cdot)$ are continuous on
  $({\Bbb T}^3)^2$ and it has a unique non-degenerate minimum at the
  point $(p_0,p_0) \in ({\Bbb T}^3)^2.$ Then by the Hadamard lemma
  \cite{Zor04} there exists a $\delta$-neighborhood of the point $p_0
  \in {\Bbb T}^3$ such that the following decomposition holds:
  $$
  \begin{aligned}
    w_2(p,q) & = m+\frac{1}{2} \left((W_1(p-p_0), p-p_0) + 2(W_2(p-p_0),
      q-p_0)+ (W_1(q-p_0), q-p_0) \right)
    \\
    & \quad +\sum\limits_{|s|+|l|=3} H_{sl}(p,q) \prod_{i=1}^3
    (p^{(i)}-p_0^{(i)})^{s_i}(q^{(i)}-p_0^{(i)})^{l_i}, \quad p,q \in
    U_\delta(p_0),
  \end{aligned}
  $$
  where
  $$
  W_1:=\left( \frac{\partial^2 w_2(p_0,p_0)}{\partial p^{(i)}
      \partial p^{(j)}} \right)_{i,j=1}^3, \quad W_2:=\left(
    \frac{\partial^2 w_2(p_0,p_0)}{\partial p^{(i)} \partial q^{(j)}}
  \right)_{i,j=1}^3,
  $$
  $$
  s=(s_1,s_2,s_3),\quad l=(l_1,l_2,l_3),\quad |s|=s_1+s_2+s_3,\quad
  s_i, l_i \in \{0,1,2,3\},\quad  i=1,2,3,
  $$
  and $H_{sl}(\cdot,\cdot)$ with $|s|+|l|=3$ are continuous
  functions in $U_\delta(p_0) \times U_\delta(p_0).$ Therefore,
  there exist positive numbers $C_1, C_2, C_3$ such that (i) and (ii)
  hold true.
\end{proof}

\section{The Weinberg type system of integral equations}
In this section we derive an analogue of the Weinberg type system
of integral equations for the eigenvectors, corresponding to the
eigenvalues of $H,$ lying on the left of $m.$

Let $\tau_{\rm ess}(H)$ be the lower bound of the essential
spectrum of $H.$ It is clear that $\Delta_\alpha(p\,;z)>0$ for all
$p \in {\Bbb T}^3$ and $z \in (-\infty; \tau_{\rm ess}(H));$ if
$\max\limits_{p \in {\Bbb T}^3} \Delta_\alpha(p\,; m)<0,$ then
$\Delta_\alpha(p\,;z)<0$ for all $p \in {\Bbb T}^3$ and $z \in
(E_{\rm max}^{(\alpha)}; m).$ So ${\rm
  sign}(\Delta_\alpha(p\,;z))$ depends on the location of $z \in
(-\infty; m) \setminus \sigma_{\rm ess}(H)$ and does not depend on
$p \in {\Bbb T}^3.$ For $z \in (-\infty; m) \setminus \sigma_{\rm
  ess}(H)$ we set $\xi_\alpha(z):={\rm sign}(\Delta_\alpha(p\,;z)).$

Let for any $z \in (-\infty; m) \setminus \sigma_{\rm ess}(H)$ the
operator $W(z)$ act in the Hilbert space ${\mathcal H}$ as a $3
\times 3$ operator matrix with entries $W_{ij}(z): {\mathcal
  H}_j \to {\mathcal H}_i,$ $i,j=0,1,2$ defined by
$$
\begin{gathered}
  W_{00}(z)g_0=(1+z-w_0)g_0, \quad W_{01}(z)g_1=-\int_{{\Bbb
      T}^3} \frac{v_0(s)g_1(s)\,ds}{\sqrt{\xi_1(z) \Delta_1(s\,;z)}},
  \\
  W_{02}(z) \equiv 0, \quad (W_{10}(z)g_0)(p)=-\frac{\xi_1(z)
    v_0(p)g_0}
  {\sqrt{\xi_1(z) \Delta_1(p\,;z)}},
  \\
  (W_{11}(z)g_1)(p)  =\frac{\xi_1(z) v_1(p)} {2\sqrt{\xi_1(z)
      \Delta_1(p\,;z)}} \int_{{\Bbb
      T}^3} \frac{v_1(s)g_1(s)\,ds}{\sqrt{\xi_1(z) \Delta_1(s\,;z)} (w_2(p,s)-z)},
  \\
  (W_{12}(z)g_2)(p)  =-\frac{\xi_1(z) v_2(p)}{\sqrt{\xi_1(z)
      \Delta_1(p\,;z)}} \int_{{\Bbb T}^3} \int_{{\Bbb
      T}^3} \frac{v_1(s)v_2(t)g_2(s,t)\,ds\,dt}{\sqrt{\xi_2(z) \Delta_2(s\,;z)}(w_2(p,s)-z)},
  \\
  (W_{20}(z)g_0)(p,q)  =-\frac{v_1(p)(W_{10}(z)g_0)(q)+
    v_1(q)(W_{10}(z)g_0)(p)}{2(w_2(p,q)-z)},
\end{gathered}
$$
\begin{align*}
  (W_{21}(z)g_1)(p,q) & =-\frac{\xi_2(z) v_1(p) v_2(q)} {2(w_2(p,q)-z)
    \sqrt{\xi_2(z) \Delta_2(p\,;z)}} \int_{{\Bbb T}^3}
  \frac{v_2(s)g_1(s)\,ds}{\sqrt{\xi_1(z)
      \Delta_1(s\,;z)}(w_2(p,s)-z)}
  \\
  & \quad -\frac{\xi_2(z) v_1(q) v_2(p)}{2(w_2(p,q)-z) \sqrt{\xi_2(z)
      \Delta_2(q\,;z)}} \int_{{\Bbb T}^3}
  \frac{v_2(s)g_1(s)\,ds}{\sqrt{\xi_1(z)
      \Delta_1(s\,;z)}(w_2(q,s)-z)}
  \\
  & \quad -\frac{v_1(p)(W_{11}(z)g_1)(q)+
    v_1(q)(W_{11}(z)g_1)(p)}{2(w_2(p,q)-z)},
\end{align*}
\begin{align*}
  (W_{22}(z)g_2)(p,q) & =  \frac{\xi_2(z) v_2(p) v_2(q)}{(w_2(p,q)-z)
    \sqrt{\xi_2(z) \Delta_2(p\,;z)}} \int_{{\Bbb T}^3}
  \int_{{\Bbb T}^3} \frac{v_2(s)v_2(t) g_2(s,t)\,ds\,dt}
  {\sqrt{\xi_2(z) \Delta_2(s\,;z)}(w_2(p,s)-z)}
  \\
  & \quad +\frac{\xi_2(z) v_2(p) v_2(q)} {(w_2(p,q)-z) \sqrt{\xi_2(z)
      \Delta_2(q\,;z)}} \int_{{\Bbb T}^3} \int_{{\Bbb
      T}^3} \frac{v_2(s)v_2(t)g_2(s,t)\,ds\,dt}{\sqrt{\xi_2(z)
      \Delta_2(s\,;z)}(w_2(q,s)-z)}
  \\
  & \quad -\frac{v_1(p)(W_{12}(z)g_2)(q)+
    v_1(q)(W_{12}(z)g_2)(p)}{2(w_2(p,q)-z)},
\end{align*}
where $g_i \in {\mathcal H}_i,$ $i=0,1,2.$

We have the following lemma.

\begin{lem}\label{LEM4}
  Let Assumption $\ref{A1}$ and part {\rm (i)} of Assumption
  $\ref{A2}$ be fulfilled. If $f \in {\mathcal H}$ is an eigenvector
  corresponding to the eigenvalue $z \in (-\infty; m) \setminus
  \sigma_{\rm ess}(H)$ of $H,$ then $f$ satisfies the Weinberg
  equation $W(z)f=f.$
\end{lem}

\begin{proof}
  Let $z \in (-\infty; m) \setminus \sigma_{\rm ess}(H)$ be an
  eigenvalue of the operator $H$ and $f=(f_0,f_1,f_2) \in {\mathcal
    H}$ be the corresponding eigenvector. Then $f_0,$ $f_1$ and $f_2$
  satisfy the system of equations
  \begin{equation}\label{Eq6}
    \begin{gathered}
      (H_{00}-z)f_0+ H_{01}f_1=0,
      \\
      (H_{10}f_0)(p)+((H_{11}-z)f_1)(p)+(H_{12}f_2)(p)=0,
      \\
      (H_{21}f_1)(p,q)+((H_{22}^0-z)f_2)(p,q)-(Vf_2)(p,q)=0.
    \end{gathered}
  \end{equation}
  Since $z<m,$ from the third equation of the system \eqref{Eq6} for
  $f_2$ we have
  \begin{equation}\label{Eq7}
    f_2(p,q)=\frac{v_2(q) \overline{f}_2(p)+v_2(p)
      \overline{f}_2(q)}{w_2(p,q)-z}-\frac{v_1(q) f_1(p)+v_1(p)
      f_1(q)}{2(w_2(p,q)-z)},
  \end{equation}
  where
  \begin{equation}\label{Eq8}
    \overline{f}_2(p)=\int_{{\Bbb T}^3}v_2(s)f_2(p,s)\,ds.
  \end{equation}

  Substituting the expression \eqref{Eq7} for $f_2$ into the second
  equation in the system \eqref{Eq6} and the equality \eqref{Eq8} and
  using Assumptions \ref{A1} and \ref{A2}, we obtain
  \begin{equation}\label{Eq9}
    \begin{gathered}
      f_0=(1+z-w_0)f_0-\int_{{\Bbb T}^3} v_0(s)f_1(s)\,ds=0,
      \\
      \Delta_1(p\,;z)f_1(p)=-v_0(p)f_0+\frac{v_1(p)}{2}\int_{{\Bbb
          T}^3} \frac{v_1(s)f_1(s)\,ds}{w_2(p,s)-z}
      -v_2(p)\int_{{\Bbb T}^3}
      \frac{v_1(s)\overline{f}_2(s)\,ds}{w_2(p,s)-z},
      \\
      \Delta_2(p\,;z)\overline{f}_2(p)=-\frac{v_1(p)}{2}\int_{{\Bbb
          T}^3} \frac{v_2(s)f_1(s)\,ds}{w_2(p,s)-z}+v_2(p)\int_{{\Bbb
          T}^3} \frac{v_2(s)\overline{f}_2(s\,)\,ds}{w_2(p,s)-z}.
    \end{gathered}
  \end{equation}
  It is clear that the inequality $\xi_\alpha(z)
  \Delta_\alpha(p\,;z)>0$ holds for all $z
  \in (-\infty; m) \setminus \sigma_{\rm ess}(H)$ and $p \in {\Bbb T}^3.$
  Therefore, the
  system of equations \eqref{Eq9} has a nontrivial solution if and
  only if the following system of equations:
  \begin{align*}
    f_0 & = W_{00}(z)f_0+W_{01}(z)f_1=0,
    \\
    f_1(p) & = (W_{10}(z)f_0)(p)+ (W_{11}(z)f_1)(p)
    \\
    & \quad -\frac{\xi_1(z) v_2(p)}{\sqrt{\xi_1(z) \Delta_1(p\,;z)}}
    \int_{{\Bbb T}^3} \frac{v_1(s)\overline{f}_2(s)\,ds}
    {\sqrt{\xi_2(z) \Delta_2(s\,;z)} (w_2(p,s)-z)},
    \\
    \overline{f}_2(p) & = -\frac{\xi_2(z) v_1(p)}{2\sqrt{\xi_2(z)
        \Delta_2(p\,;z)}} \int_{{\Bbb T}^3} \frac{v_2(s) f_1(s)\,ds}
    {\sqrt{\xi_1(z) \Delta_1(s\,;z)} (w_2(p,s)-z)}
    \\
    & \quad + \frac{\xi_2(z) v_2(p)}{\sqrt{\xi_2(z) \Delta_2(p\,;z)}}
    \int_{{\Bbb T}^3} \frac{v_2(s)\overline{f}_2(s)\,ds}
    {\sqrt{\xi_2(z) \Delta_2(s\,;z)}(w_2(p,s)-z)}
  \end{align*}
  has a nontrivial solution.

  Substituting the last expressions for $f_1$ and $\overline{f}_2$
  into the formula \eqref{Eq7} and using the equality \eqref{Eq8},
  we obtain the Weinberg equation $W(z)f=f.$
\end{proof}

Set
$$
\Sigma:=\overline{[\tau_{\rm ess}(H)-1; m] \setminus \sigma_{\rm ess}(H)}.
$$

\begin{lem}\label{LEM5}
  Let assumptions in Theorem $\ref{THM4}$ be fulfilled.
  Then the operator $W(z)$ is compact for $z \in \Sigma$ and
  the operator-valued function $W(z)$ is continuous in the uniform
  operator topology for $z \in \Sigma.$
\end{lem}

\begin{proof}
  First for the convenience using Theorem \ref{THM3} we describe the structure of the set~$\Sigma:$\\
  (i) if $\min\limits_{p \in {\Bbb T}^3}\Delta_\alpha(p\,; m)\geq 0,$ then $\Sigma=[m-1; m];$\\
  (ii) if $\min\limits_{p \in {\Bbb T}^3}\Delta_\alpha(p\,; m)<0$ and
  $\max\limits_{p \in {\Bbb T}^3}\Delta_\alpha(p\,; m)\geq 0,$ then $\Sigma=[E_{\rm min}-1; E_{\rm min}],$
  where $E_{\rm min}=\min\{E_{\rm min}^{(1)}, E_{\rm min}^{(2)}\}$ and $E_{\rm min}<m;$\\
  (iii) if $\min\limits_{p \in {\Bbb T}^3}\Delta_\alpha(p\,; m)<0,$
  $\max\limits_{p \in {\Bbb T}^3}\Delta_\alpha(p\,; m)\geq 0$ and
  $\min\limits_{p \in {\Bbb T}^3}\Delta_\beta(p\,; m)\geq 0,$ then
  $\Sigma=[E_{\rm min}^{(\alpha)}-1; E_{\rm min}^{(\alpha)}]$
  with $E_{\rm min}^{(\alpha)}<m;$\\
  (iv) if $\max\limits_{p \in {\Bbb T}^3}\Delta_\alpha(p\,; m)<0$ and
  $\min\limits_{p \in {\Bbb T}^3}\Delta_\beta(p\,; m)\geq 0,$ then
  $\Sigma=[E_{\rm min}^{(\alpha)}-1; E_{\rm min}^{(\alpha)}] \cup [E_{\rm max}^{(\alpha)}; m]$
  with $E_{\rm max}^{(\alpha)}<m;$\\
  (v) if $\max\limits_{p \in {\Bbb T}^3}\Delta_\alpha(p\,; m)<0,$
  $\min\limits_{p \in {\Bbb T}^3}\Delta_\beta(p\,; m)<0$ and
  $\max\limits_{p \in {\Bbb T}^3}\Delta_\beta(p\,; m)\geq 0,$ then
  $$
  \Sigma=
  \left \lbrace
    \begin{array}{ll}
      \qquad\quad\,\, [E_{\rm min}-1; E_{\rm min}], \qquad \quad\,\,\,\quad \mbox{if}\,\,
      E_{\rm max}^{(\alpha)} \geq E_{\rm min}^{(\beta)}, \\

      [E_{\rm min}^{(\alpha)}-1; E_{\rm min}^{(\alpha)}] \cup
      [E_{\rm max}^{(\alpha)}; E_{\rm min}^{(\beta)}], \,\,\quad \mbox{if}\,\, E_{\rm max}^{(\alpha)}< E_{\rm min}^{(\beta)}
    \end{array}\right.
  $$
  with $E_{\rm max}^{(\alpha)}, E_{\rm min}^{(\beta)}<m;$\\
  (vi) if $\max\limits_{p \in {\Bbb T}^3}\Delta_\alpha(p\,; m)<0,$
  then $\Sigma=[\tau_{\rm ess}(H)-1; m] \setminus \{(E_{\rm min}^{(1)}; E_{\rm max}^{(1)}) \cup
  (E_{\rm min}^{(2)}; E_{\rm max}^{(2)}) \}$ with $E_{\rm max}^{(\alpha)}<m.$

  We will prove the statement of the lemma for the case (vi) with $E_{\rm min}:=E_{\rm min}^{(1)}=E_{\rm min}^{(2)}$
  and $E_{\rm max}:=E_{\rm max}^{(1)}=E_{\rm max}^{(2)}.$
  Other cases can be proven in a similar.

  Let $\max\limits_{p \in {\Bbb T}^3}\Delta_\alpha(p\,; m)<0$ and Assumptions \ref{A3}, \ref{A4} be fulfilled.
  For $z \in (-\infty; m) \setminus \sigma_{\rm ess}(H)$ denote by $W(p,q,s,t;z)$ the kernel of the
  operator $W_{22}(z).$

  We have the following inequalities:
  $$
  w_2(p,q)-z \geq m-E_{\rm min}>0 \quad \mbox{for all} \quad p,q \in {\Bbb T}^3, \quad z \leq E_{\rm min};
  $$
  $$
  w_2(p,q)-z \geq (m-E_{\rm max})/2>0 \quad \mbox{for all} \quad p,q \in {\Bbb T}^3, \quad z \in
  [E_{\rm max}; (m+E_{\rm max})/2].
  $$

  Then by Assumptions \ref{A3}, \ref{A4} and Lemma \ref{LEM6}
  the function $|W(\cdot,\cdot,\cdot,\cdot;z)|$ can be estimated by
  \begin{align*}
    C_1 & \times
    \left(1\!+\!\sum_{i=1}^{n_2}\frac{\chi_{\delta}(s-p_{2i})}{|s-p_{2i}|^{\beta_{2i}/2}}\right)\\
    & \times
    \left(1\!+\!\sum_{i=1}^{n_1}\frac{\chi_{\delta}(p-p_{1i})}{|p-p_{1i}|^{\beta_{1i}/2}}
      \!+\!\sum_{i=1}^{n_1}\frac{\chi_{\delta}(q-p_{1i})}{|q-p_{1i}|^{\beta_{1i}/2}}
      \!+\!\sum_{i=1}^{n_2}\frac{\chi_{\delta}(p-p_{2i})}{|p-p_{2i}|^{\beta_{2i}/2}}
      \!+\!\sum_{i=1}^{n_2}\frac{\chi_{\delta}(q-p_{2i})}{|q-p_{2i}|^{\beta_{2i}/2}}
    \right)
  \end{align*}
  for $z \leq E_{\rm min}$ and by
  \begin{align*}
    C_2 & \times
    \left(1\!+\!\sum_{j=1}^{k_2}\frac{\chi_{\rho}(s-q_{2j})}{|s-q_{2j}|^{\gamma_{2j}/2}}\right)\\
    & \times
    \left(1\!+\!\sum_{j=1}^{k_1}\frac{\chi_{\rho}(p-q_{1j})}{|p-q_{1j}|^{\gamma_{1j}/2}}
      \!+\!\sum_{j=1}^{k_1}\frac{\chi_{\rho}(q-q_{1j})}{|q-q_{1j}|^{\gamma_{1j}/2}}
      \!+\!\sum_{j=1}^{k_2}\frac{\chi_{\rho}(p-q_{2j})}{|p-q_{2j}|^{\gamma_{2j}/2}}
      \!+\!\sum_{j=1}^{k_2}\frac{\chi_{\rho}(q-q_{2j})}{|q-q_{2j}|^{\gamma_{2j}/2}}
    \right)
  \end{align*}
  for $z\in [E_{\rm max}, (m+E_{\rm max})/2],$ where $\chi_{\delta}(\cdot)$ is the characteristic
  function of $U_\delta(\overline{0}).$

  Since $\xi_\alpha(z)=1$ for any $z \in (E_{\rm max}; m)$ and $\max\limits_{p \in {\Bbb T}^3} \Delta_\alpha(p\,; m)<0,$
  for any $z \in (E_{\rm max}; m)$ we have $\max\limits_{p \in {\Bbb T}^3} (\xi_\alpha(z)\Delta_\alpha(p\,; m))>0.$
  Therefore, Lemmas \ref{LEM2} and \ref{LEM3}
  imply that the function $|W(\cdot,\cdot,\cdot,\cdot;z)|$ can be estimated by
  \begin{align*}
    C_3 \left(1+\frac{\chi_{\delta}(p-p_0) \chi_{\delta}(q-p_0)}{|p-p_0|^2+|q-p_0|^2} \right)
    \left(1+\frac{\chi_{\delta}(p-p_0) \chi_{\delta}(s-p_0)}{|p-p_0|^2+|s-p_0|^2}
      +\frac{\chi_{\delta}(q-p_0) \chi_{\delta}(s-p_0)}{|q-p_0|^2+|s-p_0|^2}
    \right)
  \end{align*}
  for $z \in [(m+E_{\rm max})/2; m].$

  The latter three functions are square integrable on $({\Bbb T}^3)^4$ and hence
  the operator $W_{22}(z)$ is Hilbert Schmidt for any $z \in
  (-\infty; E_{\rm min}] \cup [E_{\rm max}; m].$

  A similar argument shows that the operators $W_{11}(z),$ $W_{12}(z)$ and $W_{21}(z)$
  are also Hilbert Schmidt for any $z \in \Sigma.$

  For any $z \in (-\infty; m) \setminus \sigma_{\rm ess}(H)$ the kernel function of $W_{ij}(z),$
  $i,j=1,2$ is continuous on its domain. Therefore the continuity of the operator-valued functions
  $W_{ij}(z),$ $i,j=1,2$ in the uniform operator topology for $z \in \Sigma$ follows
  from Lebesgue's dominated convergence theorem.

  Since for all $z \in \Sigma$ the operators $W_{00}(z),$ $W_{01}(z),$ $W_{10}(z)$ and $W_{20}(z)$
  are of rank 1 and continuous in the uniform operator topology for $z \in \Sigma,$ one concludes that
  $W(z)$ is compact for $z \in \Sigma$ and the operator-valued function $W(z)$ is continuous
  in the uniform operator topology for $z \in \Sigma.$
\end{proof}

\section{Proof of the main results}

In this section we prove Theorems \ref{THM3} and \ref{THM4}.

\begin{proof}[Proof of Theorem $\ref{THM3}$] Let $\alpha=1.$
  If $z_1 \in {\Bbb C} \setminus \sigma_{\rm ess}(H_1)$ is an eigenvalue of the operator $H_1$
  and $f=(f_0,f_1,f_2) \in {\mathcal H}$ is the corresponding
  eigenvector, then $f_0,$ $f_1$ and $f_2$ are satisfy the following system of
  equations:
  \begin{equation}\label{Eq10}
    \begin{gathered}
      (w_0-z_1)f_0+\int_{{\Bbb T}^3} v_0(s)f_1(s)\,ds=0,
      \\
      v_0(p)f_0+(w_1(p)-z_1)f_1(p)+\int_{{\Bbb T}^3}
      v_1(s)f_2(p,s)\,ds=0,
      \\
      \frac{1}{2}(v_1(p)f_1(q)+v_1(q)f_1(p))+(w_2(p,q)-z_1)f_2(p,q)=0.
    \end{gathered}
  \end{equation}
  Since $z_1 \not \in [m; M],$ from the third equation of the system
  \eqref{Eq10} for $f_2$ we have
  \begin{equation}\label{Eq11}
    f_2(p,q)=-\frac{v_1(p)f_1(q)+v_1(q)f_1(p)}{2(w_2(p,q)-z_1)}.
  \end{equation}

  Substituting the expression \eqref{Eq11} for $f_2$ into the second
  equation of the system \eqref{Eq10}, we obtain
  $$
  \Delta_1(p\,;z_1) f_1(p)=\frac{v_1(p)}{2}
  \int_{{\Bbb T}^3} \frac{v_1(s)f_1(s)\,ds}{w_2(p,s)-z_1}-v_0(p)f_0.
  $$

  Since $z_1 \not \in \sigma_{\rm ess}(H_1)$ the inequality $\Delta_1(p\,;z_1) \neq 0$
  holds for all $p \in {\Bbb T}^3.$ From the last equation we have
  $$
  f_1(p)=\frac{v_1(p)}{2 \Delta_1(p\,;z_1)}
  \int_{{\Bbb T}^3} \frac{v_1(s)f_1(s)\,ds}{w_2(p,s)-z_1}-
  \frac{v_0(p)f_0}{\Delta_1(p\,;z_1)}.
  $$

  The functions $v_0(\cdot),$ $v_1(\cdot),$ $w_1(\cdot)$ and
  $w_2(\cdot,q),$ $q \in {\Bbb T}^3$ are $2 \bar{\pi}$ periodic and hence
  the function $f_1(\cdot)$ is also $2 \bar{\pi}$ periodic.
  Therefore, for any fixed $p \in {\Bbb T}^3$ the
  function $f_2(p,\cdot)$ defined by \eqref{Eq11}, is $2 \bar{\pi}$ periodic. Hence this
  function satisfies the condition \eqref{Eq2}, that is, $Vf_2=0.$
  So the number $z_1 \in \sigma_{\rm disc}(H_1)$ is an eigenvalue of $H$
  with the same eigenvector $f=(f_0,f_1,f_2) \in {\mathcal H}.$
  Therefore, $\sigma_{\rm disc}(H_1) \subset \sigma_{\rm p}(H).$

  Let now $z_2 \in \sigma_{\rm disc}(H_2)$ and $g_2 \in {\mathcal H}_2$ be the
  eigenfunction corresponding to the discrete eigenvalue $z_2.$ Then similar
  analysis shows that $H_{12}g_2=0,$ which guarantee that the number $z_2 \in \sigma_{\rm disc}(H_2)$
  is an eigenvalue of $H$ and corresponding eigenvector $g$ has form $g=(0,0,g_2) \in {\mathcal H},$
  that is, $\sigma_{\rm disc}(H_2) \subset \sigma_{\rm p}(H).$
  Theorem \ref{THM3} is proved.
\end{proof}

\begin{proof}[Proof of Theorem $\ref{THM4}$]
  We prove the finiteness of the number of discrete eigenvalues located on the left of
  $m$ for the case when $\max\limits_{p \in {\Bbb T}^3}\Delta_\alpha(p\,; m)<0.$
  Other cases can be proven similarly. Suppose that the operator $H$ has an
  infinite number of discrete eigenvalues $\left( E_k \right)_{k \in {\Bbb N}} \subset (E_{\rm max}; m).$
  Then three cases are possible

  (i) $\lim\limits_{k \to \infty} E_k=m;$

  (ii) $\lim\limits_{k \to \infty} E_k=E_{\rm max};$

  (iii) there exist $\left( E_k' \right)_{k \in {\Bbb N}}, \left( E_k'' \right)_{k \in {\Bbb N}}
  \subset \left( E_k \right)_{k \in {\Bbb N}}$ such that $\lim\limits_{k \to \infty} E_k'=m$
  and $$\lim\limits_{k \to \infty} E_k''=E_{\rm max}.$$

  Let us consider the case (iii). For each $k \in {\Bbb N}$ we denote by $\varphi_k \in {\mathcal H}$
  and $\psi_k \in {\mathcal H}$ the orthonormal
  eigenvectors corresponding to the eigenvalues $E_k'$ and $E_k'',$ respectively.
  Then it follows from Lemma \ref{LEM4} that $\varphi_k=W(E_k')\varphi_k$ and $\psi_k=W(E_k'')\psi_k$
  for any $k \in {\Bbb N}.$ By virtue of Lemma \ref{LEM5}
  the operators $W(E_{\rm max}), W(m)$ are compact and $\|W(z)-W(E_{\rm max})\| \to 0$ and
  $\|W(z)-W(m)\| \to 0$ as $z \to E_{\rm max}+0$ and $z \to m-0,$ respectively.
  Therefore,
  \begin{align*}
    & 1=\|\varphi_k\|=\|W(E_k')\varphi_k\| \leq \|(W(E_k')-W(E_{\rm max}))\varphi_k\|+
    \|W(E_{\rm max})\varphi_k\| \to 0,\\
    & 1=\|\psi_k\|=\|W(E_k'')\psi_k\| \leq \|(W(E_k'')-W(m))\psi_k\|+
    \|W(m)\psi_k\| \to 0
  \end{align*}
  as $k \to \infty.$ This contradiction implies that the points $z=E_{\rm max}$ and $z=m$
  can not be limit points of the set of discrete eigenvalues of $H$ belonging to the interval
  $(E_{\rm max}; m).$ Similar arguments show that other edges of $\Sigma$ are also cannot
  be accumulation point for the set of discrete eigenvalues of $H$ smaller than $m.$
\end{proof}

The following example shows that the class of functions $w_1(\cdot),$ $v_i(\cdot),$ $i=1,2$ and
$w_2(\cdot,\cdot)$ satisfying the conditions of Theorem
$\ref{THM4}$ is nonempty.

\begin{lem}\label{Example}
  Let $$ \widehat{v}_1(p):=\!\sum\limits_{i=1}^3 c_i \cos p^{(i)},
  \quad \widehat{v}_2(p):=\!\sum\limits_{i=1}^3 d_i \cos
  (p^{(i)}/2), \quad v_\alpha(p):=\!\!\sqrt{2^{2-\alpha}
    \mu_\alpha}\, \widehat{v}_\alpha(p), \quad \alpha=1,2, $$ $$
  w_1(p) \equiv 1, \quad w_2(p,q)=\varepsilon(p)+\varepsilon(q),
  \quad \varepsilon(p)= \sum\limits_{i=1}^3 (1- \cos p^{(i)}), $$
  where $\mu_\alpha>0;$ $c_i, d_i,$ $i=1,2,3$ are arbitrary real
  numbers.

  Set
  $$
  \mu_\alpha^{(0)}:=\Bigg( \int_{{\Bbb T}^3} \frac{\widehat{v}_\alpha^2(s)\,ds}
  {\varepsilon(s)} \Bigg)^{-1}, \quad \mu_\alpha^{(1)}:=
  \Bigg( \int_{{\Bbb T}^3} \frac{\widehat{v}_\alpha^2(s)\,ds}
  {6+\varepsilon(s)} \Bigg)^{-1}.
  $$

  Then the functions $w_1(\cdot),$ $v_\alpha(\cdot),$ $\alpha=1,2$ and $w_2(\cdot,\cdot)$
  are satisfy Assumptions $\ref{A1}$, $\ref{A2}$, $\ref{A3}$, $\ref{A4}$. Moreover,

  {\rm (i)} if $0<\mu_\alpha<\mu_\alpha^{(0)},$ then $\min\limits_{p \in {\Bbb T}^3}\Delta_\alpha(p\,;
  m) > 0;$

  {\rm (ii)} if $\mu_\alpha^{(0)}< \mu_\alpha \leq \mu_\alpha^{(1)},$ then
  $\min\limits_{p \in {\Bbb T}^3}\Delta_\alpha(p\,; m)<0$ and
  $\max\limits_{p \in {\Bbb T}^3}\Delta_\alpha(p\,; m)\geq 0;$

  {\rm (iii)} if $\mu_\alpha > \mu_\alpha^{(1)},$ then
  $\max\limits_{p \in {\Bbb T}^3}\Delta_\alpha(p\,; m)<0.$
\end{lem}

\begin{proof}
  Let $g \in L_2({\Bbb T}^3)$ be as in Assumption \ref{A1}. Then we have
  $$
  \int_{{\Bbb T}^3} v_2(s)g(s)\,ds=\int_{{\Bbb T}^3}
  v_2(s+2\bar{\pi}) g(s+2\bar{\pi})\,ds=-\int_{{\Bbb
      T}^3} v_2(s)g(s)\,ds,
  $$
  which yields the equality \eqref{Eq2}, that is, Assumption \ref{A1} holds with
  $\alpha=1$ and $\beta=2.$

  From the definition of $w_2(\cdot,\cdot)$ it follows that this function has a unique zero
  non-degenerate minimum at $(\overline{0},\overline{0}) \in ({\Bbb T}^3)^2$
  and it satisfies all conditions of Assumption \ref{A2}.

  The assertions (i)--(iii) directly follow from the definition of the numbers
  $\mu_\alpha^{(0)}$ and $\mu_\alpha^{(1)}.$

  Let $\mu_\alpha^{(0)}< \mu_\alpha \leq \mu_\alpha^{(1)}.$
  We prove that the function $\Delta_\alpha(\cdot; E_{\rm min}^{(\alpha)})$ has a unique
  non-degenerate minimum at $\overline{0} \in {\Bbb T}^3.$
  Simple calculations show that $\Delta_\alpha(p; E_{\rm min}^{(\alpha)}) >
  \Delta_\alpha(\overline{0};$ $E_{\rm min}^{(\alpha)})$ for all $p \neq \overline{0}.$

  Since $E_{\rm min}^{(\alpha)} \in (-\infty, 0),$ it is clear that the function
  $\Delta_\alpha(\cdot\,; E_{\rm min}^{(\alpha)})$ is twice continuously differentiable in
  ${\bf T}^3.$ Moreover, from the equalities
  \begin{equation*}
    \begin{aligned}
      \frac{\partial^2 \Delta_\alpha(p\,; E_{\rm min}^{(\alpha)})}{\partial p^{(i)} \partial
        p^{(i)}} & =  \mu_\alpha \cos\, p^{(i)} \int \limits_{{\Bbb T}^3}
      \frac{\widetilde{v}_\alpha^2(s)\,ds}{(\varepsilon(p)+\varepsilon(s)-E_{\rm min}^{(\alpha)})^2}
      \\
      & \quad -2 \mu_\alpha (\sin\, p^{(i)})^2 \int_{{\Bbb T}^3}
      \frac{\widetilde{v}_\alpha^2(s)\,ds}{(\varepsilon(p)+\varepsilon(s)-E_{\rm min}^{(\alpha)})^3},
      \quad i=1,2,3,
      \\
      \frac{\partial^2 \Delta_\alpha(p\,; E_{\rm min}^{(\alpha)})}{\partial p^{(i)} \partial
        p^{(j)}}& = -2 \mu_\alpha \sin\, p^{(i)} \sin\, p^{(j)} \int_{{\Bbb T}^3}
      \frac{\widetilde{v}_\alpha^2(s)\,ds}{(\varepsilon(p)+\varepsilon(s)-E_{\rm min}^{(\alpha)})^3},
      \quad i \neq j, \quad i,j=1,2,3
    \end{aligned}
  \end{equation*}
  we get
  $$
  \frac{\partial^2 \Delta_\alpha(\overline{0}\,; E_{\rm
      min}^{(\alpha)})}{\partial p^{(i)}
    \partial p^{(i)}}>0, \quad \frac{\partial^2 \Delta_\alpha(\overline{0}\,; E_{\rm min}^{(\alpha)})}{\partial
    p^{(i)} \partial p^{(j)}}=0, \quad i \neq j, \quad i,j=1,2,3.
  $$

  Using these facts, one may verify that the matrix of the second order partial derivatives
  of the function $\Delta_\alpha(\cdot; E_{\rm min}^{(\alpha)})$
  at the point $p = \overline{0}$ is positive definite. Thus
  the function $\Delta_\alpha(\cdot; E_{\rm
    min}^{(\alpha)})$ has a non-degenerate minimum at the point $p = \overline{0}.$
  Then the equality $\Delta_\alpha(\overline{0}; E_{\rm
    min}^{(\alpha)})=0$ implies that there exist the numbers $\delta>0$ and
  $C>0$ such that
  $$
  |\Delta_\alpha(p\,; E_{\rm min}^{(\alpha)})|\geq C
  |p|^2,\quad p\in U_{\delta}(\overline{0}),
  $$
  that is, Assumption \ref{A3} holds with $n_\alpha=1,$
  $p_{\alpha 1}=\overline{0}$ and $\beta_{\alpha 1}=2.$

  In the case $\mu_\alpha > \mu_\alpha^{(1)}$ one can similarly show that
  there exist $\rho>0$ and
  $K>0$ such that
  $$
  |\Delta_\alpha(p\,; E_{\rm max}^{(\alpha)})|\geq K
  |p-\overline{\pi}|^2,\quad p\in U_{\rho}(\overline{\pi}),
  $$
  that is, Assumption \ref{A4} holds with $k_\alpha=1,$
  $q_{\alpha 1}=\overline{\pi}$ and $\gamma_{\alpha 1}=2.$
\end{proof}

{\it Acknowledgments.} This work was supported by the TOSCA
Erasmus Mundus grant. The author wishes to thank the University
of L'Aquila for the invitation and
hospitality.

\end{document}